\documentclass{article}
\usepackage{latexsym}
\usepackage{epsfig}
\usepackage{amsmath}
\usepackage{amssymb}
\usepackage{amsthm}
\usepackage{xcolor}
\usepackage{hyperref}
\newtheorem{theorem}{Theorem}[section]
\newtheorem{corollary}[theorem]{Corollary}
\newtheorem{lemma}[theorem]{Lemma}
\newtheorem{example}[theorem]{Example}

\newtheorem{assumption}[theorem]{Assumption}

\theoremstyle{remark}

\theoremstyle{definition}
\newtheorem{definition}[theorem]{Definition}

\DeclareMathOperator\diag{diag}

\begin{document}

\title{Optimization in complex spaces with the Mixed Newton Method}

\author{Sergei Bakhurin\thanks{%
Huawei Technologies Co.\ Ltd.,\ Moscow Department of R\&D center, 121099 Moscow, Russia
({\tt bakhurin.sergey@huawei.com}).}
\and
Roland Hildebrand\thanks{%
Univ.\ Grenoble Alpes, CNRS, Grenoble INP, LJK, 38000 Grenoble, France;\ Moscow Institute of Physics and Technology, 141701 Dolgoprudny, Russia
({\tt khildebrand.r@mipt.ru}).}
\and
Mohammad Alkousa\thanks{%
Moscow Institute of Physics and Technology, 141701 Dolgoprudny, Russia
({\tt mohammad.alkousa@phystech.edu}).} 
\and
Alexander Titov\thanks{%
Moscow Institute of Physics and Technology, 141701 Dolgoprudny, Russia
({\tt a.a.titov@phystech.edu}).} 
\and
Nikita Yudin\thanks{%
Moscow Institute of Physics and Technology, 141701 Dolgoprudny, Russia
({\tt iudin.ne@phystech.edu}). Federal Research Center "Informatics and Control" of Russian Academy of Sciences, 119333 Moscow, Russia.} 
}

\maketitle

\begin{abstract}
We propose a second-order method for unconditional minimization of functions $f(z)$ of complex arguments. We call it the Mixed Newton Method due to the use of the mixed Wirtinger derivative $\frac{\partial^2f}{\partial\bar z\partial z}$ for computation of the search direction, as opposed to the full Hessian $\frac{\partial^2f}{\partial(z,\bar z)^2}$ in the classical Newton method. The method has been developed for specific applications in wireless network communications, but its global convergence properties are shown to be superior on a more general class of functions $f$, namely sums of squares of absolute values of holomorphic functions. In particular, for such objective functions minima are surrounded by attraction basins, while the iterates are repelled from other types of critical points. We provide formulas for the asymptotic convergence rate and show that in the scalar case the method reduces to the well-known complex Newton method for the search of zeros of holomorphic functions. In this case, it exhibits generically fractal global convergence patterns.
\end{abstract}


\section{Introduction}

This work arose from efforts of practitioners working in telecommunications to reduce the computational load when tuning the parameters of models for digital filter design. It was empirically observed that the method proposed in this paper not only needed less time per iteration, but showed superior global performance, not suffering the slowdown in the vicinity of saddle points which the ordinary Newton method experienced. The paper presents the results of a theoretical study conducted to give the observed phenomena a theoretical underpinning.

Wireless network communication systems are using digital linear and nonlinear filters, which allow keeping specified characteristics of the system if operating conditions change. The adaptation procedure can be reduced to an optimization problem on the complex-valued parameters of the filters, which can be realized as a software subroutine or as hardware intellectual property cores (IP-cores) on mobile phone chipsets.

The optimization problem can be formulated as follows. For $D \subset \mathbb C^n$ a domain and $g_j: D \to \mathbb C$ holomorphic functions, $j = 1,\dots,m$, find
\begin{equation} \label{f_definition}
\min_{z \in D} f(z) := \sum_{j=1}^m |g_j(z)|^2,
\end{equation}
which is in general a non-convex optimization task. The function $f(z)$ is real-valued and can be described as a function $f(z, \bar z)$ of two independent variables $z$ and $\bar z$, in which it is, respectively, holomorphic and anti-holomorphic when the value of the other variable is fixed.

Different approaches have been tried to solve problem \eqref{f_definition}, or in general the problem of minimization of a real-valued function of complex variables, numerically. One possibility is to consider the domain $D$ as a subset of real space $\mathbb R^{2n}$ and to apply usual iterative minimization methods such as gradient descent or the Newton method.
However, in practice these methods show undesirable behavior, they tend to get stuck in the neighbourhood of saddle points of $f$ or even converge to such points. In addition, Newton's method uses the  Hessian matrix
 \begin{equation} \label{bahurin_h_definition}
H =  \left(\begin{array}{cc}
\frac{\partial^2 f(z, \bar z)}{\partial z \partial z}     &  \frac{\partial^2 f(z, \bar z)}{\partial \bar z \partial z}\\
\frac{\partial^2 f(z, \bar z)}{\partial z \partial \bar z}     &  \frac{\partial^2 f(z, \bar z)}{\partial \bar z \partial \bar z}
\end{array}\right),
\end{equation}
which contains every second-order partial derivative with respect to $z$ and $\bar z$. 

As an alternative to Newton's method, a simplified algorithm has been proposed and implemented by the first author
empirically. It shows superior global convergence properties while having lower computational complexity.
This \emph{Mixed Newton Method} (MNM) uses only one block of mixed second derivatives from the Hessian~\eqref{bahurin_h_definition} and generates iterates according to the formula
\begin{equation} \label{MNMiteration}
z_{k+1} = z_k - \left( \frac{\partial^2 f}{\partial\bar z\partial z} \right)^{-1} \frac{\partial f}{\partial\bar z},
\end{equation}
where the derivatives are to be understood as \emph{Wirtinger derivatives}, evaluated at the current iterate $z = z_k$, and the gradient is interpreted as a column vector. The use of the Hessian~\eqref{bahurin_h_definition} is less efficient since it requires the calculation and processing of a matrix in $\mathbb{C}^{2n \times 2n}$. A short introduction to Wirtinger derivatives and a formal definition are provided in the Appendix in Section \ref{sec:Wirtinger}.

Wirtinger derivatives have been used previously for solving optimization problems in complex spaces. In \cite{SorberBarelLathauwer12} classical first- and second-order optimization methods such as L-BFGS have been rewritten in terms of complex derivatives. This allows for faster computation in comparison to a decomposition of the complex variable in a real and imaginary part. However, the Newton method has been applied to complex spaces already in \cite{YanFan00}, and the use of complex derivatives for signal processing applications has been proposed in \cite{Bos94}, see also \cite{AdaliSchreier14},\cite{Candan19}, and \cite{SchreierScharf}.

In this paper, we study the theoretical properties of the Mixed Newton Method and present experimental results. We provide a theoretical ground for its observed preference to converge to minima rather than saddle points. We also construct counter-examples against further going conjectures which have been made based on first observational evidence, namely convergence to minima a.s. or exclusive convergence to global minima. A separate subject of study in this paper is the effect of symmetry groups. This is motivated by the aforementioned application in wireless communication, in which the holomorphic functions are often multi-linear in two or more subsets of complex variables.

The property of converging to minima and being repelled by saddle points can be seen as a global property of the method. Convergence to a minimum is preferred even if we start far away from the global or even a local minimum. Nothing can be said in general about the volume of the different attraction basins, however. This question is further complicated by the fact that these basins are fractal in general.

The remainder of the paper is structured as follows. In Section \ref{sec:asymptoticMixed}, we prove the convergence of the MNM to minima and its repulsion from saddle points (Theorem \ref{thm:minima_attractive}).  The main idea which leads to this result is to link the spectrum of the real Hessian at a critical point to the spectrum of the coefficient matrix of a linear dynamical system that asymptotically describes the behavior of the MNM in the vicinity of this critical point. Section \ref{sec:asymptoticMixed} references Appendix \ref{sec:Wirtinger}, which contains an overview over Wirtinger calculus. In Section \ref{sec:special}, we specify the derived formulas for the asymptotic convergence rates to two special cases, first where the holomorphic functions $g_j$ are complex-affine (Lemma \ref{lem:complex_affine}), and second where $f = |g|^2$ is the squared modulus of a single scalar holomorphic function (Corollary \ref{cor:scalar_quadratic}). This provides also a transparent interpretation of the MNM. Namely, it can be seen as approximating the objective function $f$ in \eqref{f_definition} by truncating the Taylor polynomials of the functions $g_j$ to first order and then going to the minimum of the resulting convex quadratic approximation. In Subsection \ref{sec:square_example} we present an example where the global dynamics induced by the MNM has an easy interpretation. In Section \ref{sec:symmetry} we study the impact of symmetry groups on the MNM iterates. We show that in this case, it is necessary to introduce a regularization for numerical stability, but this regularization has no impact on the sequence of iterates generated by the method (Lemma \ref{lem:regulMNM}). In Section \ref{sec:counterexamples} we provide examples of functions that  refute further going conjectures which could have been made on the basis of initial observational evidence. In particular, we present functions where the MNM may converge to local minima (Subsection \ref{sec:localMinima}), and functions where the MNM may not converge at all,  instead of approaching a stable periodic cycle (Subsection \ref{sec:periodicCycle}). In Section \ref{sec:applications} we sketch the type of models used in telecommunications to which the proposed method is applied and showed superior global convergence properties. Finally, we summarize our results in Section \ref{sec:conclusion}.

\section{Behavior  at critical points} \label{sec:asymptoticMixed}

In this section we derive the asymptotic dynamics of the MNM in a neighbourhood of a critical point of the objective function $f$, given by \eqref{f_definition}. Here critical point is understood in the classical sense that the real gradient of $f$ vanishes. In view of \eqref{Wirtinger_conjugate} this is equivalent to the condition $\frac{\partial f}{\partial\bar z} = 0$. Let $z^* \in \mathbb C^n$ be a critical point of $f$, and denote $\delta_k = z_k - z^*$. We shall derive the asymptotic behavior of the sequence of mismatches $\delta_k$ in the regime $\|\delta_k\| \to 0$.

First, we make the MNM iteration more explicit for the class of objective functions we consider. According to \eqref{modulus_derivatives}, we get the formulas
\begin{equation} \label{derivatives_f}
\frac{\partial f}{\partial\bar z} = \sum_{j=1}^m g_j\overline{g_j'}, \qquad \frac{\partial^2f}{\partial\bar z\partial z} = \sum_{j=1}^m \overline{g_j'}(g_j')^T.
\end{equation}
We recall that $g_j'$ is interpreted as a column vector composed of the holomorphic derivatives of $g_j$ with respect to the different elements of the vector $z$ of complex arguments of the function.

In order for the MNM iteration \eqref{MNMiteration} to be defined, we have to ensure that the mixed Hessian $\frac{\partial^2f}{\partial\bar z\partial z}$ is an invertible matrix. The explicit formula above represents this matrix as a sum of rank 1 positive semi-definite complex Hermitian matrices with images generated by the vectors $\overline{g_j'}$. We shall make the following assumption:

\begin{assumption} \label{assumption_nondegenerate}
In the neighbourhood of the critical point $z^*$ the derivatives $g_j'$ span the whole complex vector space $\mathbb C^n$.
\end{assumption}

\smallskip

This guarantees that the complex Hermitian matrix $\frac{\partial^2f}{\partial\bar z\partial z}$ is positive definite.

\medskip

We shall now consider the behavior of the MNM near the stationary point $z^*$. The stationarity condition is given by
\[ \left. \frac{\partial f}{\partial\bar z} \right|_{z = z^*} = \sum_{j=1}^m g_j(z^*)\overline{g_j'(z^*)} = 0.
\]
We then obtain that the mismatch between the iterate and the stationary point evolves according to
\[ \delta_{k+1} = \delta_k - \left( \sum_{j=1}^m \overline{g_j'(z_k)}(g_j'(z_k))^T \right)^{-1} \sum_{j=1}^m g_j(z_k)\overline{g_j'(z_k)}.
\]
Developing the functions $g_j$ and its derivatives around $z = z^*$, we obtain
\[ g_j(z_k) = g_j(z^*) + \delta_k^Tg_j'(z^*) + O(\|\delta_k\|^2), \qquad g_j'(z_k) = g_j'(z^*) + g_j''(z^*)\delta_k + O(\|\delta_k\|^2).
\]
Hence
\begin{align*}
\delta_{k+1} &= \delta_k - \left( \sum_{j=1}^m (\overline{g_j'(z^*)} + O(\|\delta_k\|))(g_j'(z^*) + O(\|\delta_k\|))^T \right)^{-1} \cdot \\
& \cdot \sum_{j=1}^m (g_j(z^*) + \delta_k^Tg'_j(z^*) + O(\|\delta_k\|^2))(\overline{g_j'(z^*)} + \overline{g''(z^*)}\overline{\delta_k} + O(\|\delta_k\|^2)) \\
&= \delta_k - \left( \sum_{j=1}^m (\overline{g_j'(z^*)})(g_j'(z^*))^T \right)^{-1} \sum_{j=1}^m \left( g_j(z^*)\overline{g''(z^*)}\overline{\delta_k} + \overline{g_j'(z^*)}(g'_j(z^*))^T\delta_k + O(\|\delta_k\|^2) \right) \\
&= - \left( \sum_{j=1}^m (\overline{g_j'(z^*)})(g_j'(z^*))^T \right)^{-1} \left( \sum_{j=1}^m g_j(z^*)\overline{g''(z^*)} \right)\overline{\delta_k} + O(\|\delta_k\|^2).
\end{align*}

Therefore we obtain
\begin{equation} \label{asymptotic_Mixed}
\delta_{k+1} = -B^{-1}A\overline{\delta_k} + O(\|\delta_k\|^2), \qquad B = \sum_{j=1}^m (\overline{g_j'(z^*)})(g_j'(z^*))^T, \quad A = \sum_{j=1}^m g_j(z^*)\overline{g''(z^*)}.
\end{equation}
Note that $B = B^*$ is positive definite by Assumption \ref{assumption_nondegenerate}, while $A = A^T$ is symmetric.

When computing the MNM iterate we have to solve a linear system with coefficient matrix
\[ \sum_{j=1}^m (\overline{g_j'(z_k)})(g_j'(z_k))^T = B + O(\|\delta_k\|).
\]
This is accomplished by finding a Cholesky factor of this matrix and subsequently solving two triangular linear systems. Hence after making an iteration an approximation of the lower-triangular Cholesky factor $L$ of $B = LL^*$ is available.

In the study of the Newton method it is convenient to use the Hessian of the objective as a metric to measure the length of vectors \cite{NesNem94}, in particular, of the mismatch $\delta_k$. We, therefore, introduce the scaled mismatch
\[ \eta_k = L^*\delta_k,
\]
such that $\delta_k^*B\delta_k = \|\eta_k\|^2$. In terms of $\eta_k$ rule \eqref{asymptotic_Mixed} becomes
\begin{equation} \label{asymptotic_eta}
\eta_{k+1} = L^*\delta_{k+1} = -L^*B^{-1}A\overline{L^{-*}\eta_k} + O(\|\delta_k\|^2) = -L^{-1}AL^{-T}\overline{\eta_k} + O(\|\eta_k\|^2) = S\overline{\eta_k} + O(\|\eta_k\|^2).
\end{equation}
Hence the convergence rate is determined by the singular values of the complex symmetric matrix
\[ S = -L^{-1}AL^{-T} = S^T.
\]

\begin{lemma} \label{lem:singular_crit}
Assume above notations and Assumption \ref{assumption_nondegenerate}. Then the following assertions hold:
\begin{itemize}
\item[(a)] if $\sigma_{\max}(S) < 1$, then the sequence $z_k$ of iterates tends to $z^*$ if the initial point $z_0$ is close enough to $z^*$,
\item[(b)] if $\sigma_{\max}(S) > 1$, then the sequence $z_k$ is repulsed from $z^*$ for almost all initial values $z_0$,
\item[(c)] in the first case and under the additional assumption $S \not= 0$, for almost all initial values $z_0$ the convergence rate is linear, and the distance to $z^*$ decreases asymptotically by a factor of $\sigma_{\max}(S)$ per iteration.
\end{itemize}
\end{lemma}

\begin{proof}
Let $S = U\Sigma U^T$ be the singular value decomposition of $S$. Here $U$ is unitary, and $\Sigma = \diag(\sigma)$, $\sigma$ being the vector of singular values of $S$. Let us decompose $\eta_k$ in the basis of singular vectors of $S$, $\eta_k = Ud_k$. Here the vector $d_k = U^*\eta_k$ contains the coefficients of $\eta_k$ in this basis. We then obtain from \eqref{asymptotic_eta} that
\begin{equation} \label{asymptotics_d}
d_{k+1} = U^*\eta_{k+1} = U^*S\overline{\eta_k} + O(\|\eta_k\|^2) = U^*U\Sigma U^T\bar U\overline{d_k} + O(\|d_k\|^2) = \Sigma\overline{d_k} + O(\|d_k\|^2).
\end{equation}

Thus the elements of $d_k$ are conjugated and multiplied by the corresponding singular value of $S$ at each iteration, up to quadratic error terms. Assertion (a) now readily follows, as in this case, all elements of $d_k$ tend to zero if their initial values are small enough.

The dynamics of the MNM are dominated by the dynamics of the first component of $d_k$, which corresponds to the maximal singular value $\sigma_{\max}(S)$. The modulus of this component is multiplied by $\sigma_{\max}(S)$ at each iteration, up to higher-order terms, while the other components are multiplied by smaller constants.

Suppose now that $\sigma_{\max}(S) > 1$. Let $V \subset \mathbb C^n$ be the complex subspace spanned by the singular vectors of $S$ which have singular values not exceeding 1. If the dynamics is ideally linear, $d_{k+1} = U^*\eta_{k+1} = \Sigma\overline{d_k}$, then for every initial point $z_0 \not\in V$, and hence for almost every initial point, the sequence of iterates diverges. The addition of the analytic higher-order terms distorts the subspace $V$ of non-diverging initial points into some real-analytic variety in the neighbourhood of $z^*$, but the set of initial points outside this variety is still open and dense. This proves (b).

The assertions in (c) are proven similarly to (b), with $V \subset \mathbb C^n$ defined as the subspace spanned by the singular vectors associated to singular values $\sigma$ which are strictly smaller than $\sigma_{\max}(S)$. Linear convergence is evident from \eqref{asymptotics_d}.
\end{proof}

The lemma hence yields a criterion of convergence to the critical point $z^*$. Namely, the point is attractive if $\sigma_{\max}(S) < 1$, and repulsive if $\sigma_{\max}(S) > 1$. We shall now link this criterion to the nature of the critical point, i.e., whether it is a (non-degenerate) minimum, saddle point, or maximum of $f$. This in turn depends on the signature of the real Hessian of $f$ at $z^*$. In the case of a degenerate minimum, no claims can be made with the techniques employed in this paper.

\begin{lemma} \label{lem:spectra}
Assume above notations and Assumption \ref{assumption_nondegenerate}. The real Hessian of $f$ at $z = z^*$ has the same signature as the permuted block-diagonal matrix $\begin{pmatrix} I & -\Sigma \\ -\Sigma & I \end{pmatrix}$ with eigenvalues $1 \pm \sigma_j(S)$, where $\sigma_j(S)$ are the singular values of $S$.
\end{lemma}

\begin{proof}
The real Hessian $H_{\mathbb R}$ is a real symmetric matrix of second partial derivatives of $f$ with respect to $Re\,z$ and $Im\,z$ (see \eqref{realGradHess}). From \eqref{cr} we get
\[ \begin{pmatrix} \frac{\partial^2f}{\partial z^2} & \frac{\partial^2f}{\partial z\partial\bar z} \\ \frac{\partial^2f}{\partial \bar z\partial z} & \frac{\partial^2f}{\partial\bar z^2} \end{pmatrix} = \frac14 \begin{pmatrix} I & -iI \\ I & iI \end{pmatrix} H_{\mathbb R} \begin{pmatrix} I & -iI \\ I & iI \end{pmatrix}^T.
\]
Hence $H_{\mathbb R}$ has the same signature as the complex Hermitian matrix
\[ \begin{pmatrix} \bar B & \bar A \\ A & B \end{pmatrix} = \begin{pmatrix} \frac{\partial^2f}{\partial z\partial\bar z} & \frac{\partial^2f}{\partial z^2} \\ \frac{\partial^2f}{\partial\bar z^2} & \frac{\partial^2f}{\partial \bar z\partial z} \end{pmatrix} = \frac14 \begin{pmatrix} I & -iI \\ I & iI \end{pmatrix} \begin{pmatrix} \frac{\partial^2f}{\partial Re\,z^2} & \frac{\partial^2f}{\partial Re\,z\partial Im\,z} \\ \frac{\partial^2f}{\partial Im\,z\partial Re\,z} & \frac{\partial^2f}{\partial Im\,z^2} \end{pmatrix} \begin{pmatrix} I & -iI \\ I & iI \end{pmatrix}^*,
\]
which is obtained from the previous one by interchanging the two-column blocks. Here the matrices $A, B$ are defined in \eqref{asymptotic_Mixed}. Recall that all derivatives are evaluated at the critical point $z = z^*$.

Recall that the lower-triangular matrix $L$ was defined above as the Cholesky factor of the positive definite matrix $B$. We get
\[ \begin{pmatrix} \bar B & \bar A \\ A & B \end{pmatrix} = \begin{pmatrix} \bar L & 0 \\ 0 & L \end{pmatrix} \begin{pmatrix} I & \bar L^{-1}\bar AL^{-*} \\ L^{-1}A\bar L^{-*} & I \end{pmatrix} \begin{pmatrix} \bar L & 0 \\ 0 & L \end{pmatrix}^* = \begin{pmatrix} \bar L & 0 \\ 0 & L \end{pmatrix} \begin{pmatrix} I & -\bar S \\ -S & I \end{pmatrix} \begin{pmatrix} \bar L & 0 \\ 0 & L \end{pmatrix}^*,
\]
and the real Hessian has the same signature as the matrix
\[ \begin{pmatrix} I & -\bar S \\ -S & I \end{pmatrix} = \begin{pmatrix} I & -\bar U\Sigma U^* \\ -U\Sigma U^T & I \end{pmatrix} = \begin{pmatrix} \bar U & 0 \\ 0 & U \end{pmatrix} \begin{pmatrix} I & -\Sigma \\ -\Sigma & I \end{pmatrix} \begin{pmatrix} \bar U & 0 \\ 0 & U \end{pmatrix}^*.
\]
This matrix in turn has the same signature as the central factor $\begin{pmatrix} I & -\Sigma \\ -\Sigma & I \end{pmatrix}$, which by a permutation of rows and columns can be brought to a block-diagonal form with $2 \times 2$ blocks $\begin{pmatrix} 1 & -\sigma_j(S) \\ -\sigma_j(S) & 1 \end{pmatrix}$, which have eigenvalues $1 \pm \sigma_j(S)$.
\end{proof}

As a direct consequence we obtain the following result.

\begin{corollary}
Assume Assumption \ref{assumption_nondegenerate}. Let $z^*$ be a critical point of $f$, and let $S$ be the matrix defined as above. Let $k_+$ be the number of singular values of $S$ strictly greater than 1, $k_0$ the number of singular values equal to 1, and $k_-$ the number of singular values strictly smaller than 1. Then the real Hessian of $f$ at $z = z^*$ has $n + k_+$ positive eigenvalues, $k_0$ zero eigenvalues, and $k_-$ negative eigenvalues.

In particular,
\begin{itemize}
\item if $\sigma_{\max}(S) < 1$, then $z^*$ is a strict local minimum of $f$ with positive definite Hessian,
\item if $\sigma_{\max}(S) > 1$, then $z^*$ is a saddle point of $f$ with both positive and negative eigenvalues of the Hessian, and the number of positive eigenvalues of the Hessian is at least as large as the number of negative eigenvalues,
\item if $z^*$ is a degenerate minimum (with rank-deficient real Hessian), then $\sigma_{\max}(S) = 1$,
\item the function $f$ has no local maxima.
\end{itemize}
\end{corollary}

Combining this corollary with Lemma \ref{lem:singular_crit}, we obtain the following result.

\begin{theorem} \label{thm:minima_attractive}
Assume above notations and Assumption \ref{assumption_nondegenerate}.
\begin{itemize}
\item if $z^*$ is a strict minimum of $f$ with positive definite Hessian, then $\sigma_{\max}(S) < 1$ and the minimum is attractive for the MNM in some neighbourhood of $z^*$,
\item if $z^*$ is a saddle point of $f$ with the Hessian at $z^*$ having negative eigenvalues, then $\sigma_{\max}(S) > 1$ and the saddle point is repulsive for the MNM for almost all initial points. 
\end{itemize}
\end{theorem}

This explains the observed global behavior  of the MNM, which is attracted to minima and repelled from critical points of other types. Let us remark that at every critical point $z^*$ of $f$ the Hessian $H_{\mathbb R}$ has at least $n$ strictly positive eigenvalues.

\section{Special cases} \label{sec:special}

In this section we study two special cases when the matrix $A$ defined in \eqref{asymptotic_Mixed} vanishes.

In the first case the holomorphic functions $g_j$ defining the objective $f$ are complex-affine, and hence $g_j'' = 0$. In the second case $f$ is the squared modulus of a complex-scalar holomorphic function. The first special case sheds light on the philosophy behind the MNM, especially when compared to the ordinary Newton method, and the precise use it makes of the complex structure of the problem. The second special case yields convergence behaviors which are different from those observed in the general case, and serves as a precaution not to draw too far-going conclusions from observed phenomena when studying these simplest examples.

\subsection{Sums of squares of complex-affine functions}

Let us assume that the holomorphic functions $g_j$ are given by complex-affine expressions
\[ g_j(z) = \langle a_j,z \rangle + b_j,
\]
where $a_j \in \mathbb C^n$, $b_j \in \mathbb C$, $j = 1,\dots,m$. Here the scalar product is the usual hermitian one, $\langle y,z \rangle = y^*z$, where $y,z$ are complex columns vectors, and $y^*$ denotes the complex conjugate transpose of $y$.

In \cite{HannaSimaan85} one squared modulus of a complex-linear function has been minimized under a linear constraint, by providing an analytic expression for the solution.

\begin{lemma} \label{lem:affineComplex}
The function
\[ f(z) = \sum_{j=1}^m |\langle a_j,z \rangle + b_j|^2
\]
is convex.

It is strictly convex if and only if the vectors $a_j$ span the whole space $\mathbb C^n$, i.e., if Assumption \ref{assumption_nondegenerate} holds.
\end{lemma}

\begin{proof}
For the convexity of $f$ it suffices to show that every summand $|\langle a_j,z \rangle + b_j|^2$ is convex.

If $a_j = 0$, then $|\langle a_j,z \rangle + b_j|^2 = |b_j|^2$ is a constant.

Suppose that $a_j \not= 0$. Then the complex-affine function $\langle a_j,z \rangle + b_j$ is constant on complex-linear subspaces of complex co-dimension 1, or real co-dimension 2. The same holds for its squared modulus. On 1-dimensional complex-linear subspaces which are transversal to the level subsets of $\langle a_j,z \rangle + b_j$ this function is a complex-affine scalar, whose squared modulus is a non-degenerate convex quadratic function. Hence on $\mathbb R^{2n} \sim \mathbb C^n$ the summand $|\langle a_j,z \rangle + b_j|^2$ is a convex quadratic function with signature $(++0\dots0)$.

This proves the first assertion.

Now if the linear functionals $a_j$ have a common kernel vector, then along this vector the function $f$ is constant and hence not strictly convex. On the other hand, if the $a_j$ span the whole space, then along any direction at least one of the complex-affine functions $\langle a_j,z \rangle + b_j$ is non-constant and hence its squared modulus strictly convex.

This proves the second assertion.
\end{proof}

For the holomorphic derivatives we have $\overline{g_j'} = a_j$, and by virtue of \eqref{derivatives_f}
\[ \frac{\partial f}{\partial\bar z} = \sum_{j=1}^m (\langle a_j,z \rangle + b_j)a_j, \qquad \frac{\partial^2f}{\partial\bar z\partial z} = \sum_{j=1}^m a_ja_j^*.
\]
We obtain the following consequence.

\begin{corollary}
The function $f(z) = \sum_{j=1}^m |\langle a_j,z \rangle + b_j|^2$ has a unique minimum on $\mathbb C^n$ if and only if the mixed Hessian $\frac{\partial^2f}{\partial\bar z\partial z}$ is invertible, and hence the mixed Newton iteration is well-defined.
\end{corollary}

\begin{proof}
The matrix $\sum_{j=1}^m a_ja_j^*$ is non-singular if and only if the vectors $a_j$ span the whole space $\mathbb C^n$ and by virtue of Lemma \ref{lem:affineComplex} if and only if $f$ is strictly convex. In this case it clearly has a unique minimum as a strictly convex quadratic function.

On the other hand, if $f$ is not strictly convex, then the matrix $\sum_{j=1}^m a_ja_j^*$ is singular, hence has a kernel vector, along which $f$ is constant. Therefore its minimum is achieved on an affine subspace of positive dimension and is not unique.
\end{proof}

We now come to the main result concerning this special case.

\begin{lemma} \label{lem:complex_affine}
Suppose that the vectors $a_j$ span the whole space. Then the Mixed Newton Method reaches the unique minimum of $f$ in a single iteration regardless of the initial point.
\end{lemma}

\begin{proof}
Let $z_0$ be an arbitrary initial point. Then the next iterate is given by
\[ z_1 = z_0 - \left( \sum_{j=1}^m a_ja_j^* \right)^{-1} \sum_{j=1}^m (\langle a_j,z_0 \rangle + b_j)a_j = z_0 - \sum_{j=1}^m (\langle a_j,z_0 \rangle + b_j) \left( \sum_{l=1}^m a_la_l^* \right)^{-1} a_j.
\]
This yields
\[ \langle a_k,z_1 \rangle + b_k = \langle a_k,z_0 \rangle - \sum_{j=1}^m (\langle a_j,z_0 \rangle + b_j) a_k^* \left( \sum_{l=1}^m a_la_l^* \right)^{-1} a_j + b_k = \langle a_k,z_0 \rangle + b_k - \sum_{j=1}^m (\langle a_j,z_0 \rangle + b_j) A_{kj},
\]
where we denoted $A_{kj} = a_k^* \left( \sum_{l=1}^m a_la_l^* \right)^{-1} a_j$. However, the matrix $A = (A_{kj})_{k, j = 1, \dots, m}$ is nothing else than the projector on the image of the matrix $(a_1, \dots, a_m)^*$, and hence we have the relation $\sum_{k=1}^m A_{kj}a_k = a_j$ for all $j = 1, \dots, m$. It follows that
\begin{align*}
\left. \frac{\partial f}{\partial\bar z} \right|_{z = z_1} &= \sum_{k=1}^m (\langle a_k,z_1 \rangle + b_k)a_k = \sum_{k=1}^m \left( (\langle a_k,z_0 \rangle + b_k - \sum_{j=1}^m (\langle a_j,z_0 \rangle + b_j) A_{jk})a_k \right) \\ &= \sum_{k=1}^m \left( (\langle a_k,z_0 \rangle + b_k \right)a_k - \sum_{j=1}^m \left( (\langle a_j,z_0 \rangle + b_j) a_j \right) = 0.
\end{align*}
Hence $z_1$ is a critical point of $f$. However, this function is strictly convex quadratic with a unique minimum, which is hence the unique critical point and must coincide with $z_1$.
\end{proof}

This yields a transparent interpretation of the mixed Newton step in the general case. Note that by virtue of \eqref{derivatives_f} the next iterate $z_{k+1}$ depends only on the values $g_j(z_k)$ and the derivatives $g_j'(z_k)$ of the functions $g_j$ at the current iterate. Hence $z_{k+1}$ remains the same if we replace each $g_j$ by its first order Taylor approximation. But then $z_{k+1}$ can be interpreted as the minimum of the resulting approximation of $f$.

Hence the MNM constructs and minimizes a quadratic approximation of $f$ by using only first-order information on the $g_j$, while the ordinary Newton method constructs and minimizes a full-blown second-order Taylor approximation of $f$.

\subsection{Squared modulus of a scalar holomorphic function}

We now turn to the second special case. Let $g(z)$ be a scalar-valued holomorphic function, and let $f(z) = |g(z)|^2$.

If $g$ is not constant, then $\log\,|g(z)| = Re\,\log\,g(z)$ is a non-constant harmonic function, defined everywhere on the domain of definition of $g$ except the zeros of $g$. Since a non-constant harmonic function cannot have local minima, the minima of $f$ can be located only at the zeros of $g$. On the other hand, every zero of $g$ is clearly a local, even a global minimum of $f$. Thus the problem of minimizing $f$ is equivalent to the problem of finding the zeros of $g$.

The iteration of the MNM is by virtue of \eqref{derivatives_f} given by
\begin{equation} \label{scalar_iterate}
z_{k+1} = z_k - \left( \overline{g'(z_k)}g'(z_k) \right)^{-1} g(z_k)\overline{g'(z_k)} = z_k - \frac{g(z_k)}{g'(z_k)}.
\end{equation}
The similarity with an ordinary Newton iteration is not only superficial. Indeed, the method encodes an ordinary Newton iteration for the search of a zero of $g(z)$, interpreted as a vector field on $D \subset \mathbb C$, the domain of definition of $g$. 

The Newton method, applied to the search of a zero of a holomorphic function, has been studied, e.g., in \cite{PeitgenSaupeHaeseler84} or in \cite[Section 1.9]{Beardon}, but the problem has its roots in the work of Cayley \cite{Cayley1879}. We have the following result.

\begin{corollary} \cite[Section 1.9]{Beardon} \label{cor:scalar_quadratic}
Let $z^*$ be a zero of $g$ and suppose $g'(z^*) \not= 0$. Then the MNM converges quadratically to $z^*$ if the initial point is close enough.
\end{corollary}

This is in contrast to the generally linear convergence behavior  of the MNM on sums of several squares of holomorphic functions. However, linear convergence is recovered if the zero of $g$ is degenerated. Indeed, let $g(z) = c(z-z^*)^j + O(|z-z^*|^{j+1})$ for some $c \in \mathbb C \setminus \{0\}$ and $j > 1$. Then $g'(z) = jc(z-z^*)^{j-1} + O(|z-z^*|^j)$ and hence
\begin{align*}
z_{k+1} - z^* &= z_k - z^* - \frac{c(z_k-z^*)^j + O(|z_k-z^*|^{j+1})}{jc(z_k-z^*)^{j-1} + O(|z_k-z^*|^j)} = z_k - z^* - \frac{z_k - z^* + O(|z_k-z^*|^2)}{j + O(|z_k-z^*|)} \\ &= (1 - j^{-1})(z_k - z^*) + O(|z_k-z^*|^2).
\end{align*}
The convergence is therefore linear with rate $1 - j^{-1}$, i.e., the higher the degeneration of the zero, the slower the convergence.

\subsection{Explicit example} \label{sec:square_example}

In this section we provide an example where the dynamics of the MNM can be globally expressed by a simple analytic iteration, namely $w \mapsto w^2$ for a suitable coordinate $w$. This example enters in the framework of the second special case described in the previous section. In this special case the MNM iteration itself is a holomorphic map, and hence enjoys the larger symmetry group of conformal isomorphisms. This group can be used to transform the iteration into a particularly simple map. The example has been taken from \cite{PeitgenSaupeHaeseler84}.

Consider the cost function
\[ f(z) = |g(z)|^2, \qquad g(z) = z^2-a,
\]
where $a$ is a non-zero complex number.

It is clear that the minima of this cost function lie at the roots of $g$, i.e., at $z^* = \pm\sqrt{a}$. On the other hand, the critical points which are not minima are given by the relation $g'(z) = 0$, yielding the unique degenerate saddle point $\hat z = 0$.

{\lemma The map $z_0 \mapsto z_k$, defined by $k$-fold application of the mixed Newton iteration, is rational with degrees $d_N = 2^k$, $d_D = 2^k-1$ of the numerator and the denominator. }

\begin{proof}
We have $g'(z) = 2z$. By virtue of \eqref{scalar_iterate} one iteration of the MNM is given by the map
\begin{equation} \label{example_dynamics}
z_k \mapsto z_{k+1} = z_k - \frac{z_k^2-a}{2z_k} = \frac{z_k^2 + a}{2z_k}.
\end{equation}
This proves the assertion of the lemma for $k = 1$.

The assertion for general $k$ is proven by induction. Suppose that $z_{k-1} = \frac{N(z_0)}{D(z_0)}$, where $N,D$ are polynomials of degrees $2^{k-1}$, $2^{k-1}-1$, respectively, with no common roots. We then get
\begin{equation} \label{rational_iterate}
z_k = \frac{z_{k-1}^2+a}{2z_{k-1}} = \frac{N^2 + aD^2}{2ND}.
\end{equation}
Since every root of $2ND$ must be either a root of $N$ or a root of $D$, the polynomials $N^2 + aD^2,2ND$ have again no common roots. Clearly their degrees equal $2 \cdot 2^{k-1} = 2^k$ and $2^{k-1} + (2^{k-1}-1) = 2^k - 1$, respectively, completing the proof.
\end{proof}

We may extend the definition of the map $z \to \frac{z^2 + a}{2z}$ to $z = 0$ and $z = \infty$, defining the image of both points to be the infinitely far point $\infty$. Then the MNM defines a rational dynamical system on the Riemann sphere $\mathbb S = \mathbb C \cup \{\infty\}$. Including $\infty$ this system has three fixed points, of which $\infty$ is repelling, and the other two attracting.

{\corollary The finite fixed points of the operator $z_0 \mapsto z_k$ are given by the roots of a polynomial of degree $2^k$. }

\begin{proof}
The fixed points are found by resolving the equation $z_k = z_0$ with respect to $z_0$.

From \eqref{rational_iterate} it follows that the leading coefficient of the numerator $N$ of the rational function $z_k = z_k(z_0) = \frac{N(z_0)}{D(z_0)}$ equals 1, because it is squared at each step. On the other hand, by virtue of \eqref{rational_iterate} the leading coefficient of the denominator polynomial $D$ is multiplied by 2 at each step and hence equals $2^k$. The leading coefficient of the numerator polynomial $z_0D(z_0) - N(z_0)$ of the difference $z_0 - z_k$ hence equals $2^k - 1$ and is non-zero.

But the fixed points of the map $z_0 \mapsto z_k$ are the roots of this numerator polynomial, and because the denominator $D$ has no common roots with $N$, it cannot have common roots with the numerator $z_0D - N$ of the difference $z_0 - z_k$ either.
\end{proof}

It follows that the number of fixed points of the operator $z_0 \mapsto z_k$ grows exponentially with $k$, defining finite periodic orbits of the mixed Newton iterate. Every non-trivial such orbit consists of points which do not converge to a minimum of the objective.

However, all these points lie on a line through the origin which is orthogonal to the segment linking the two minima of $f$. This line actually divides the attraction basins of the minima. Moreover, the dynamics inside the attraction basins can be described by the following coordinate transformation.

{\lemma \cite[p.~12]{PeitgenSaupeHaeseler84} Let $\beta = \sqrt{a}$ be one of the minima of $f$, and let $H = \{ z \in \mathbb C \mid Re(z\cdot\bar\beta) > 0 \}$ be an open half-plane containing $\beta$. Introduce the coordinate $w = \frac{z - \beta}{z + \beta}$. Then $w$ runs through the open unit disc $\mathbb D$ if $z$ runs through $H$, and the dynamics in terms of the $w$ coordinate is given by $w \mapsto w^2$. }

\begin{proof}
Indeed, the boundary $\{ w \in \mathbb C \mid |w| = 1 \}$ of the unit disc corresponds to the set $\{ z \mid |z - \beta| = |z + \beta| \}$, which is easily verified to be identical with the line $\{ z \in \mathbb C \mid Re(z\cdot\bar\beta) = 0 \}$. Further, the point $z = \beta \in H$ corresponds to $w = 0 \in \mathbb D$, so $H$ is mapped to the interior of $\mathbb D$.

Let $w_0$ be the value of $w$ corresponding to $z_0$, and $w_1$ the value corresponding to $z_1$. From the formulas
\[ w_0 = \frac{z_0 - \beta}{z_0 + \beta}, \qquad z_1 = \frac{z_0^2 + \beta^2}{2z_0}, \qquad w_1 = \frac{z_1 - \beta}{z_1 + \beta}
\]
we easily obtain the desired relation $w_1 = w_0^2$.
\end{proof}

A similar behavior  can be observed in the opposite open half-plane $H_- = \{ z \in \mathbb C \mid Re(z\cdot\bar\beta) < 0 \}$. All periodic orbits hence lie on the line $l = \{ z \in \mathbb C \mid Re(z\cdot\bar\beta) = 0 \}$ separating the two half-planes.

It follows that all points in $H$ lie in the attraction basin of the minimum $\beta$, the convergence to the minimum is quadratic, and in order to reach a precision of $|w_k| = \varepsilon$, a number of
\[ k = \log_2\frac{\log\varepsilon}{\log|w_0|} = \log_2\frac{\log\varepsilon}{\log\frac{|z_0 - \beta|}{|z_0 + \beta|}}
\]
steps is needed.

The quantity $\left|\log\frac{|z_0 - \beta|}{|z_0 + \beta|}\right|$ can hence be seen as a measure of close-ness to the line $l$. While the convergence rate to the minima is quadratic, the recession rate from the line $l$ is linear. This is best seen when we consider the dynamics in the $w$ coordinate. We have $|w_{k+1}| = |w_k|^2$, and if $|w_k|$ is close to 1, the distance from $|w_{k+1}|$ to 1 will be approximately two times as large as the distance from $|w_k|$ to 1. The contour lines of $\left|\log\frac{|z_0 - \beta|}{|z_0 + \beta|}\right|$ are depicted on Fig.~\ref{scalar_fig}.

\begin{figure}[h]
\centering
\includegraphics[width=13cm,height=12cm]{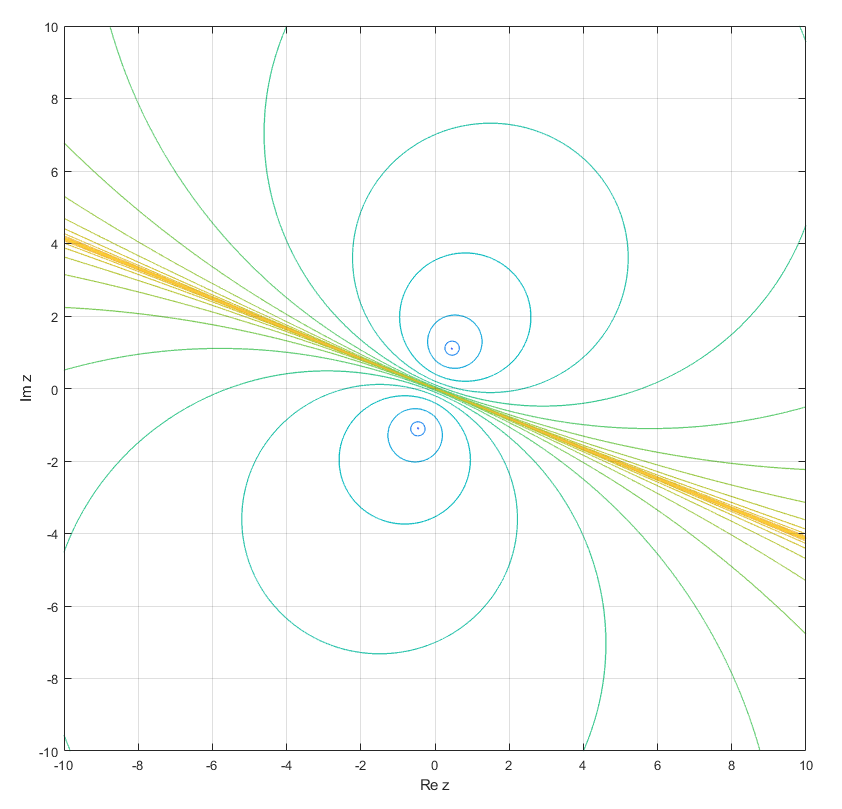}
\caption{Lines whose points are equidistant from the minima in terms of the number of iterates, for $a = -1+i$.}
\label{scalar_fig}
\end{figure}

We hence come to the following conclusion:

Although the dynamics is no more asymptotically linear in the neighbourhood of the saddle point $\hat z = 0$, it is still asymptotically linear in the neighbourhood of the line $l$ as a whole, approximately doubling the distance from this line at each iteration. Being close to the infinitely far point is no worse in terms of iteration count than being close to the origin or to any other point of the line $l$.

The dynamics on the line $l$ itself is isomorphic to a dynamical system on the unit circle which doubles the argument at each step. In particular, the periodic orbits are dense in $l$, and there are $2^k-1$ periodic orbits of length $k$ (including those having a length dividing $k$).

\section{Symmetry groups} \label{sec:symmetry}

The presence of continuous symmetry groups may inhibit the validity of Assumption \ref{assumption_nondegenerate}. Namely, if the minima of $f$ are not isolated but rather form orbits of the group action, the real Hessian at the minima must necessarily have zero eigenvalues. In the case which is relevant for practice, namely when the functions $g_j$ are multi-linear up to an additive constant, the mixed Hessian $B = \frac{\partial^2f}{\partial\bar z\partial z}$ is singular, as the following result shows.

\begin{lemma} \label{lem:kernel_grad_Hess}
Suppose that the argument vector $z$ consists of two sub-vectors $z_1,z_2$, and the functions $g_j$ satisfy the relation
\[ g_j(\zeta z_1,\zeta^{-1}z_2) = g_j(z_1,z_2) \qquad \forall\ \zeta \in \mathbb C_{\times},
\]
where $\mathbb C_{\times} = \mathbb C \setminus \{0\}$ is the multiplicative group of complex scalars. Then the vector $\xi = \begin{pmatrix} z_1 \\ -z_2 \end{pmatrix}$ is both orthogonal to the gradient $\frac{\partial f}{\partial\bar z}$ and contained in the kernel of the mixed Hessian $\frac{\partial^2f}{\partial\bar z\partial z}$. In particular, the latter is singular.
\end{lemma}

\begin{proof}
Clearly we have
\[ f(\zeta z_1,\zeta^{-1}z_2) = f(z_1,z_2) \qquad \forall\ \zeta \in \mathbb C_{\times}.
\]
For $\zeta = 1 + \varepsilon$ with $\varepsilon \in \mathbb C$ small we have
\[ \begin{pmatrix} \zeta z_1 \\ \zeta^{-1}z_2 \end{pmatrix} = \begin{pmatrix} z_1 \\ z_2 \end{pmatrix} + \begin{pmatrix} \varepsilon z_1 \\ -\varepsilon z_2 \end{pmatrix} + O(\varepsilon^2) = z + \varepsilon\xi + O(\varepsilon^2).
\]
Hence the derivative of $f$ in the direction $a\xi$ vanishes for all $a \in \mathbb C$. In particular, we have
\[ 2\,Re\,\left( a\frac{\partial f}{\partial z}^T\xi \right) = a\frac{\partial f}{\partial z}^T\xi + \bar a\frac{\partial f}{\partial \bar z}^T\bar\xi = 0 \qquad \forall\ a \in \mathbb C,
\]
which implies $\frac{\partial f}{\partial z}^T\xi = \langle \frac{\partial f}{\partial\bar z},\xi \rangle = 0$. Here we used \eqref{Wirtinger_conjugate}. This proves the first assertion.

Differentiating the relation $\frac{\partial f}{\partial z}^T\xi = 0$ with respect to $\bar z$ and noting that the vector field $\xi$ is holomorphic and hence $\frac{\partial\xi}{\partial\bar z} = 0$, we obtain the second assertion $\frac{\partial^2f}{\partial\bar z\partial z}\xi = 0$.
\end{proof}

In the arising particular situation that the kernel of the mixed Hessian $\frac{\partial^2f}{\partial\bar z\partial z}$ is orthogonal to the gradient $\frac{\partial f}{\partial\bar z}$, we can circumvent the degeneracy of the former by virtue of the following result.

\begin{lemma} \label{lem:regul}
Let $B \succeq 0$ be a positive semi-definite complex Hermitian matrix with kernel $V \subset \mathbb C^n$ of dimension $k$. Let further $u \in \mathbb C^n$ be a vector which is orthogonal to all kernel vectors $v \in V$. Let $F \in \mathbb C^{n \times k}$ be a matrix with columns forming a basis of the subspace $V$. Then for every positive definite complex Hermitian $k \times k$ matrix $P$ the sum $B + FPF^*$ is invertible, and the product $(B + FPF^*)^{-1}u$ does not depend on $P$.
\end{lemma}

\begin{proof}
Since the images of $B$ and $F$ are orthogonal, we find a coordinate system in $\mathbb C^n$, by applying a unitary transformation, such that $B,F$ decompose in blocks
\[ B = \begin{pmatrix} B' & 0 \\ 0 & 0 \end{pmatrix}, \qquad F = \begin{pmatrix} 0 \\ F' \end{pmatrix}
\]
for some $B'$ complex Hermitian positive definite of size $n - k$, and some $F' \in \mathbb C^{k \times k}$ invertible. Then we have
\[ B + FPF^* = \begin{pmatrix} B' & 0 \\ 0 & F'P(F')^* \end{pmatrix} \succ 0
\]
for every $P \succ 0$. In particular, the sum in question is invertible. Moreover, the vector $u$ is in the image of $B$, and hence can also be decomposed as $u = \begin{pmatrix} u' \\ 0 \end{pmatrix}$ for some $u' \in \mathbb C^{n-k}$. Then we obtain
\[ (B + FPF^*)^{-1}u = \begin{pmatrix} B' & 0 \\ 0 & F'P(F')^* \end{pmatrix}^{-1}\begin{pmatrix} u' \\ 0 \end{pmatrix} = \begin{pmatrix} (B')^{-1}u' \\ 0 \end{pmatrix},
\]
which is independent of $P$.
\end{proof}

In particular, the limit $\lim_{P \to 0}(B + FPF^*)^{-1}u$ is well-defined for every sequence of positive definite matrices $P$ tending to zero, and is given by the product $(B + FPF^*)^{-1}u$ for arbitrary $P \succ 0$.

We obtain the following consequence.

\begin{lemma} \label{lem:regulMNM}
Suppose the vector $z$ is divided into sub-vectors $z_1,\dots,z_l$, and every holomorphic function $g_j(z)$ is of the form
\[ g_j(z) = \tilde g_j(z) + c_j,
\]
where $c_j \in \mathbb C$ are constants and $\tilde g_j(z) = \tilde g_j(z_1,\dots,z_l)$ is multi-linear in the sub-vectors of $z$.
Assume further that the kernel of the mixed Hessian $\frac{\partial^2f}{\partial\bar z\partial z}$ is spanned by the columns of the matrix
\[ \Xi = \begin{pmatrix} z_1 & z_1 & \dots & z_1 \\ -z_2 & 0 & \dots 0 \\ 0 & -z_3 & \dots & 0 \\ \vdots & \vdots & \vdots & \vdots \\ 0 & 0 & \dots & -z_l \end{pmatrix}.
\]

Then the MNM can be regularized by adding the term $\Xi\Xi^*$ to the mixed Hessian. More concretely, performing the iteration according to the rule
\[ z \leftarrow z - \left( \frac{\partial^2f}{\partial\bar z\partial z} + \Xi\Xi^* \right)^{-1} \frac{\partial f}{\partial\bar z}
\]
leads to the same result as using the iteration
\[ z \leftarrow z - \lim_{\varepsilon \to 0} \left( \frac{\partial^2f}{\partial\bar z\partial z} + \varepsilon\Xi\Xi^* \right)^{-1} \frac{\partial f}{\partial\bar z}.
\]
\end{lemma}

\begin{proof}
First note that the columns of $\Xi$ are both in the kernel of $\frac{\partial^2f}{\partial\bar z\partial z}$ and orthogonal to $\frac{\partial f}{\partial\bar z}$. The proof is by an obvious adaptation of the proof of Lemma \ref{lem:kernel_grad_Hess}. The assertion then follows from Lemma \ref{lem:regul}.
\end{proof}

In the sequel, if we speak of the MNM for bi-linear or multi-linear models, we always mean the regularized MNM as presented in Lemma \ref{lem:regulMNM}.

The results from Section \ref{sec:asymptoticMixed} with some adaptations carry over to the situation considered in this section. Instead of the matrices $A, B$ we have to consider their restrictions to the orthogonal complement of the column space of $\Xi$, and the matrix $S$ whose singular values determine the convergence behavior has to be computed from those restrictions. The convergence (or repulsion) has to be understood not to a particular critical point $z^*$, but to its orbit under the symmetry group. 

\section{Counter examples} \label{sec:counterexamples}

\subsection{Uniqueness of minima} \label{sec:localMinima}

As was shown in the previous section, on the class of problems given by \eqref{f_definition} the MNM is converging to minima while avoiding saddle points. In this section, we provide examples of such problems with several minima, not all of which are global. All examples are sums of squares of scalar holomorphic functions.

\medskip

\begin{example}\label{ex2minima}
Consider the objective function
\[ f(z) = \left|\exp((2+0.2i)z+5-2i)\right|^2 + \left|\log((1-i)z-1-0.1i)\right|^2 + \left|\frac{2z+5}{3z+4}\right|^2.
\]
For this function, we have two minima, see Fig. \ref{basinsex2minima}.
\begin{figure}[htp]
		\begin{center}
			\includegraphics[width=0.75\linewidth]{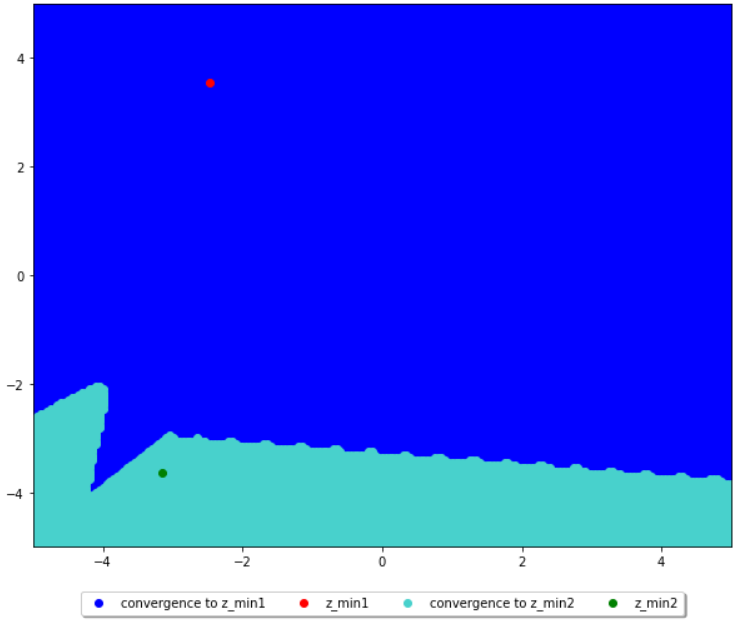}
		\end{center}
		\caption{Attraction basins of the two minima for the objective function in Example \ref{ex2minima}.}
\label{basinsex2minima}
\end{figure}
\end{example}

\begin{example}\label{ex3minima}
Consider the objective function
\[ f(z) = \left|\exp((-2+5i)z-2-3i)\right|^2 + \left|\log((0.1+0.5i)z-1-i)\right|^2 + \left|\frac{2z+5}{3z+4}\right|^2.
\]
For this function, we have three minima, see Fig. \ref{basinsex3minima}.
\begin{figure}[htp]
		\begin{center}
			\includegraphics[width=0.80\linewidth]{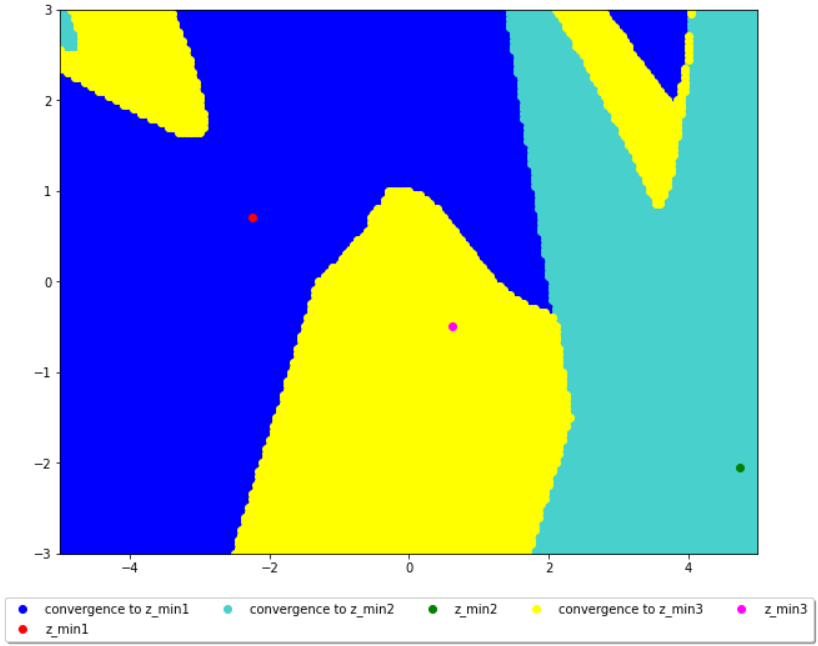}
		\end{center}
		\caption{Attraction basins of the three minima for the objective function in Example \ref{ex3minima}.}
\label{basinsex3minima}
\end{figure}
\end{example}

\begin{example}\label{ex4minima}
Consider the objective function
\[ f(z) = \left|\exp((-2+5i)z-2-3i)\right|^2 + \left|\log((0.1+0.5i)z-1-i)\right|^2 + \left|\frac{(2+i)z+1+2i}{(-1+i)z^2+(-4+2i)z-3-i}\right|^2.
\]
For this function, we have four minima, see Fig. \ref{basinsex4minima}.
\begin{figure}[htp]
		\begin{center}
			\includegraphics[width=0.88\linewidth]{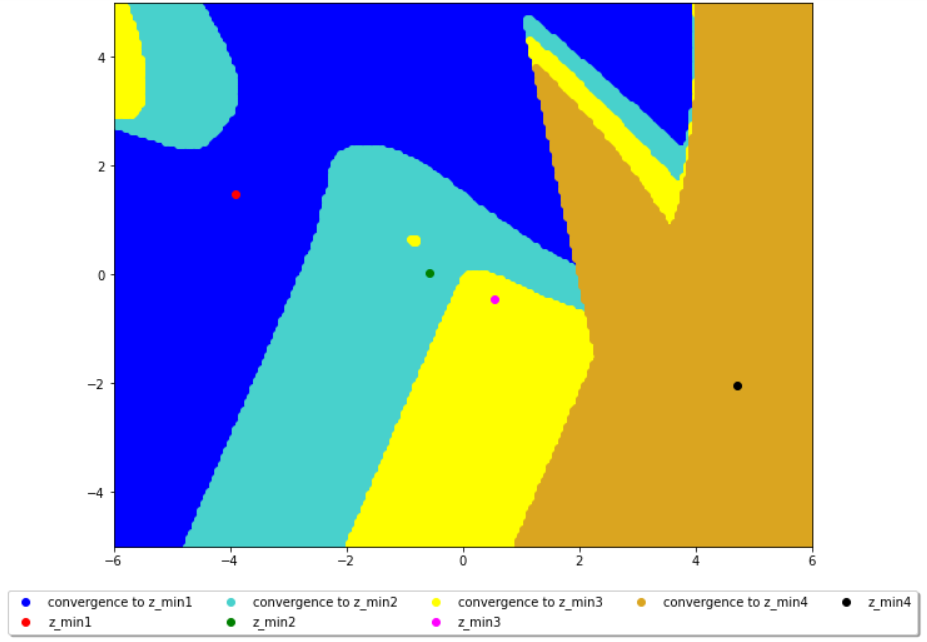}
		\end{center}
		\caption{Attraction basins of the four minima for the objective function in Example \ref{ex4minima}.}
\label{basinsex4minima}
\end{figure}
\end{example}


\subsection{Attractive cycles} \label{sec:periodicCycle}

In this section, we consider more examples that fall  in the framework of the second special case described in Section \ref{sec:special}. Other examples with pictures of attraction basins can be found in \cite[Section 3]{Wiersma16}, but here we concentrate on the presence of non-trivial attracting periodic cycles.

The fact that the mixed Newton iterates are attracted to minima and repelled from saddle points does not yet imply that every initial point in a generic position gives rise to a sequence of iterates that converges to a minimum. Namely, the iterates may also accumulate to a periodic cycle or behave chaotically. In this section, we confirm the former possibility by presenting examples of functions on which the MNM has attractive periodic cycles.

Let us generalize the example of Section \ref{sec:square_example}. Consider a non-linear \emph{rational} function $g(z)$ and the corresponding objective $f(z) = |g(z)|^2$. By virtue of \eqref{scalar_iterate}
the mixed Newton step corresponds to the application of a rational function. It is convenient to include the infinitely far point into consideration, and to study the dynamics on the Riemann sphere $\mathbb S$.

{\lemma Let the function $g$ be meromorphic in a neighbourhood of $z = \infty$. Then $z = \infty$ is a fixed point of the Mixed Newton Method. If the function $g$ remains finite at infinity, then the point $\infty$ is attractive, otherwise, it is repulsive. }

\begin{proof}
Let the two leading terms of the expansion of $g$ around $z = \infty$ be $c_mz^m + c_{m'}z^{m'}$, $m,m' \in \mathbb Z$, $m > m'$. Then \eqref{scalar_iterate} yields
\[ z_{k+1} = z_k - \frac{c_mz_k^m + c_{m'}z_k^{m'} + o(z_k^{m'})}{c_mmz_k^{m-1} + c_{m'}m'z_k^{m'-1} + o(z_k^{m'-1})} = \left\{ \begin{array}{rcl} \frac{m-1}{m}z_k - o(z_k), &\quad& m \not= 0, \\ -\frac{c_0}{c_{m'}m'}z_k^{1-m'} + O(z_k), && m = 0. \end{array} \right.
\]
If $m > 1$, then $\left|\frac{m-1}{m}\right| < 1$, and $\infty$ is repulsive with a linear rate of recession. If $m = 1$, then $z_{k+1} = o(z_k)$, and $\infty$ is repulsive with a super-linear rate of recession. If $m = 0$, then $m' < 0$, and $z_{k+1} = O(z_k^{1-m'})$ grows at least as fast as $z_k^2$. Hence $\infty$ is attractive with a super-linear rate of convergence. If $m < 0$, then $\left|\frac{m-1}{m}\right| > 1$, and $\infty$ is attractive with a linear rate of convergence.
\end{proof}

Note that if $m < 0$, then the function $g$ has a zero at $z = \infty$, and hence $f$ has a local minimum. This minimum is attractive like finite minima. However, if $m = 0$, then $z = \infty$ is neither a minimum nor a maximum of $f$. Nevertheless the method "wrongly" converges to that point.

\medskip

The theory of iterates of rational maps (on general Riemann surfaces) is quite advanced. A central place in the description of the dynamics is played by the \emph{Julia set} and the \emph{Fatou sets}.

\begin{definition}
Let $h: \mathbb S \to \mathbb S$ be a non-constant holomorphic map. The \emph{Fatou set} $F(h)$ of $h$ is the largest open subset of $\mathbb S$ on which the family of iterates of $h$, $\{ h^{\circ n} \}_{n \geq 0}$, is normal (has compact closure in the topology of uniform convergence on compact subsets). The \emph{Julia set} $J(h)$ is defined to be the complement of $F(h)$.
\end{definition}

Here $h^{\circ n}$ is the operator of $n$-fold application of $h$. In other words, the Fatou set is the largest subset of $\mathbb S$ such that for a given point in $F(h)$ any of its nearby points experience a qualitatively similar behavior  when the map $h$ is iterated on them. On the contrary, the Julia set separates domains containing points whose iterates behave in a qualitatively different manner. In the example of Section \ref{sec:square_example}, the line $l$ (including $\infty$) is the Julia set, while the open half-planes $H,H_-$ are the two connected components of the Fatou set.

The general theory \cite{Beardon},\cite{PeitgenSaupeHaeseler84} yields the following assertions:
\begin{itemize}
\item The minima of $f$ together with their attraction basins lie in the Fatou set, while saddle points lie in the Julia set.
\item The backwards orbit of any point $z_0 \in J(h)$ is dense in $J(h)$.
\item The Julia set is the closure of the union of repelling periodic points.
\item Connected components of $F(h)$ are mapped to each other, and the dynamics of these components becomes eventually periodic.
\end{itemize}

Let $H$ be a Fatou component which is mapped to itself by the power $h^{\circ p}$. Then the dynamics of $h^{\circ p}$ is of one of the following four types:
\begin{itemize}
\item an attractive basin if $h^{\circ p}$ has a fixed point that attracts all points of $H$ under iteration;
\item a parabolic basin if some point on the boundary of $H$ attracts all points of $H$;
\item a Siegel disk if the dynamical system defined by $h^{\circ p}$ is conformally isomorphic to an irrational rotation of the unit disk;
\item a Herman ring if this dynamical system is conformally isomorphic to an irrational rotation of some annulus.
\end{itemize}

In the context of the MNM the main question to analyze is the following:

\medskip

\emph{Do there exist periodic Fatou components other than attractive basins with a period $p = 1$?}

\medskip

An affirmative answer to this question implies that the only points which do not converge to minima of $f$ or the infinitely far point (if that happens) are the points of the Julia set. A negative answer implies that there exists an open subset of $\mathbb C$ all whose points do not converge to a minimum, nor to infinity.

\medskip

\emph{Example:} The following example shows that the answer to the question is, in general, negative. Consider the rational function
\[ g(z) = \frac{(-10 + 4i)z^2 + 4z + 16 - 15i}{3z^2 + (-23 + 3i)z - 7 + 3i}.
\]
This leads to the iterate
\[ z_{k+1} = \frac{(30 - 12i)z_k^4 - 24z_k^3 + (6 + 65i)z_k^2 + (646 - 786i)z_k + 67 - 153i}{(206 - 122i)z_k^2 + (20 - 26i)z_k + 295 - 381i}.
\]
The function $f = |g|^2$ has two minima with corresponding attraction basins. However, a starting point $z_0 \in \mathbb C$ in generic position can yield also a sequence that
\begin{itemize}
\item converges to the infinitely far point;
\item converges to the periodic cycle $\{-1.893587299330874 + 3.118941827800031i,-1.623493978443789 - 2.198560522966791i \}$ of length 2.
\end{itemize}
The Fatou set consists of infinitely many connection components, and the dynamics of these components is such that any sequence of iterates finally ends up either in one of three fixed components (corresponding to the two minima and the infinitely remote point) or in the periodic cycle of length two.

The attraction basins of the four attracting objects are depicted in Fig.~\ref{fractal_basins}. In Fig.~\ref{fractal_convergence} we depict the number of iterations which is needed from a given initial point to reach an $\varepsilon$-neighbourhood of one of the attracting objects (in the case of the infinitely far point, to surpass an absolute value of $\varepsilon^{-1}$). It can be seen that convergence to the periodic orbit is slower (in fact, linear, while convergence to the fixed points of the map is quadratic). The boundary between the attraction basins is the Julia set of the rational map.

\begin{figure}
\centering
\includegraphics[width=13cm,height=12.5cm]{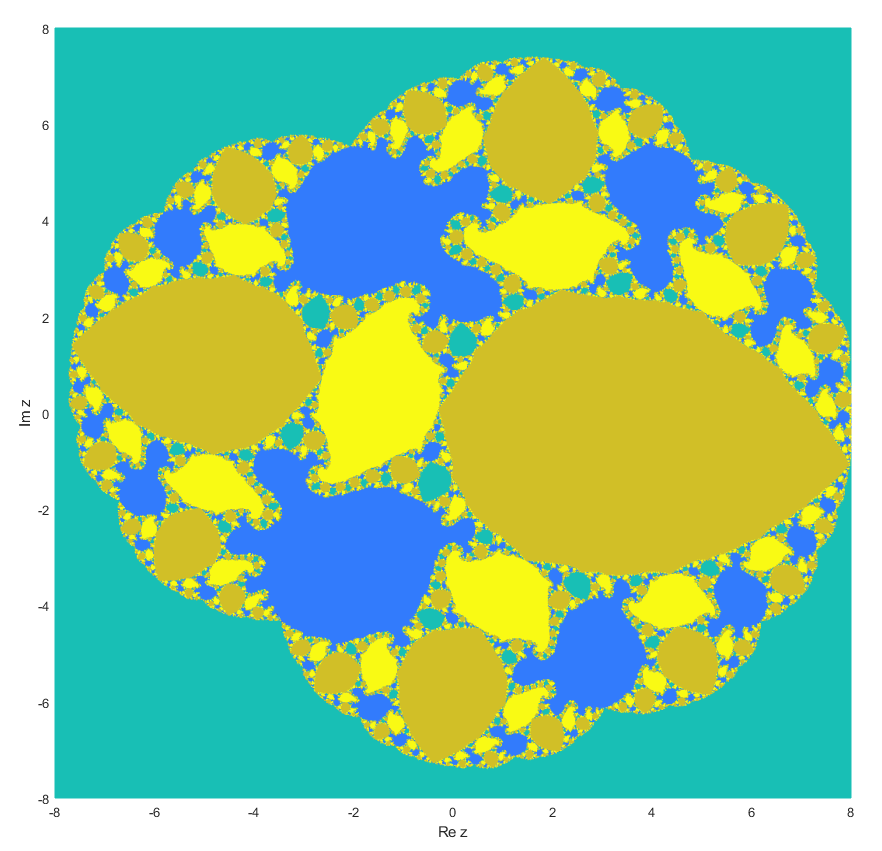}
\caption{Attraction basins of the two minima (yellow-brown and yellow), the infinitely far point (light blue) and the periodic 2-cycle (blue).}
\label{fractal_basins}
\end{figure}

\begin{figure}[h]
\centering
\includegraphics[width=13cm,height=12cm]{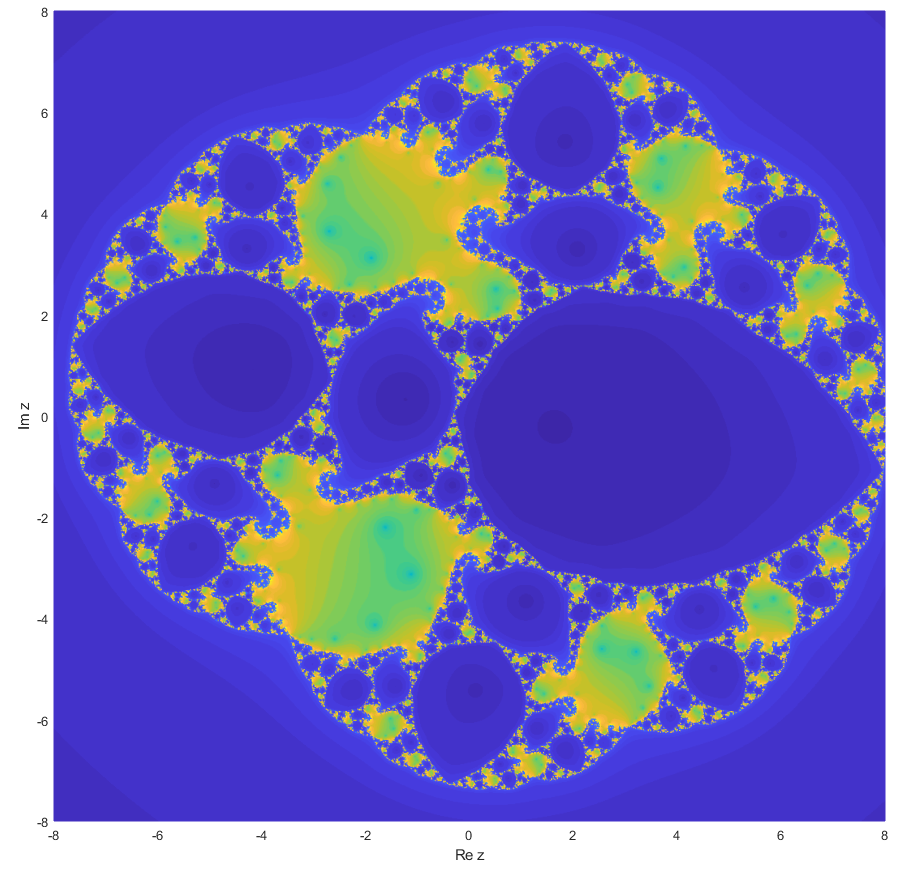}
\caption{Number of iterations needed for $\varepsilon$-convergence to an attracting object.}
\label{fractal_convergence}
\end{figure}

\medskip

\emph{Example:} Let us give another example with a polynomial function
\[ g(z) = z^3 + (1.33 + 0.81i)z^2 + (1.38 + 1.20i)z + 0.82 - 0.03i
\]
This function is from the class studied in \cite{CurryGarnettSullivan83}, where attracting cycles of different lengths have been found for cubic polynomials $g(z)$.

The mixed Newton iterate for the above example takes the form
\[ z_{k+1} = \frac{2z_k^3 + (1.33 + 0.81i)z_k^2 - 0.82 + 0.03i}{3z_k^2 + (2.66 + 1.62i)z_k + 1.38 + 1.20i}.
\]

The function $f = |g|^2$ has three minima with corresponding attraction basins. However, there exists also the attractive periodic cycle $\{-0.429935304964516 - 0.280763328984984i,-0.604967059812480 + 0.456563910615763i\}$ of length 2 with its own attraction basin. A starting point in this basin will not converge to any of the minima, and this behavior  is robust against small perturbations.

The attraction basins of the four attracting objects are depicted in Fig.~\ref{fractal_cubic}, left. In the right part of Fig.~\ref{fractal_cubic}, we depict the number of iterations which is needed from a given initial point to reach an $\varepsilon$-neighbourhood of one of the attracting objects. Again the convergence to the periodic orbit is linear, while convergence to the fixed points of the map is quadratic.

\begin{figure}
\centering
\includegraphics[width=14.27cm,height=8.14cm]{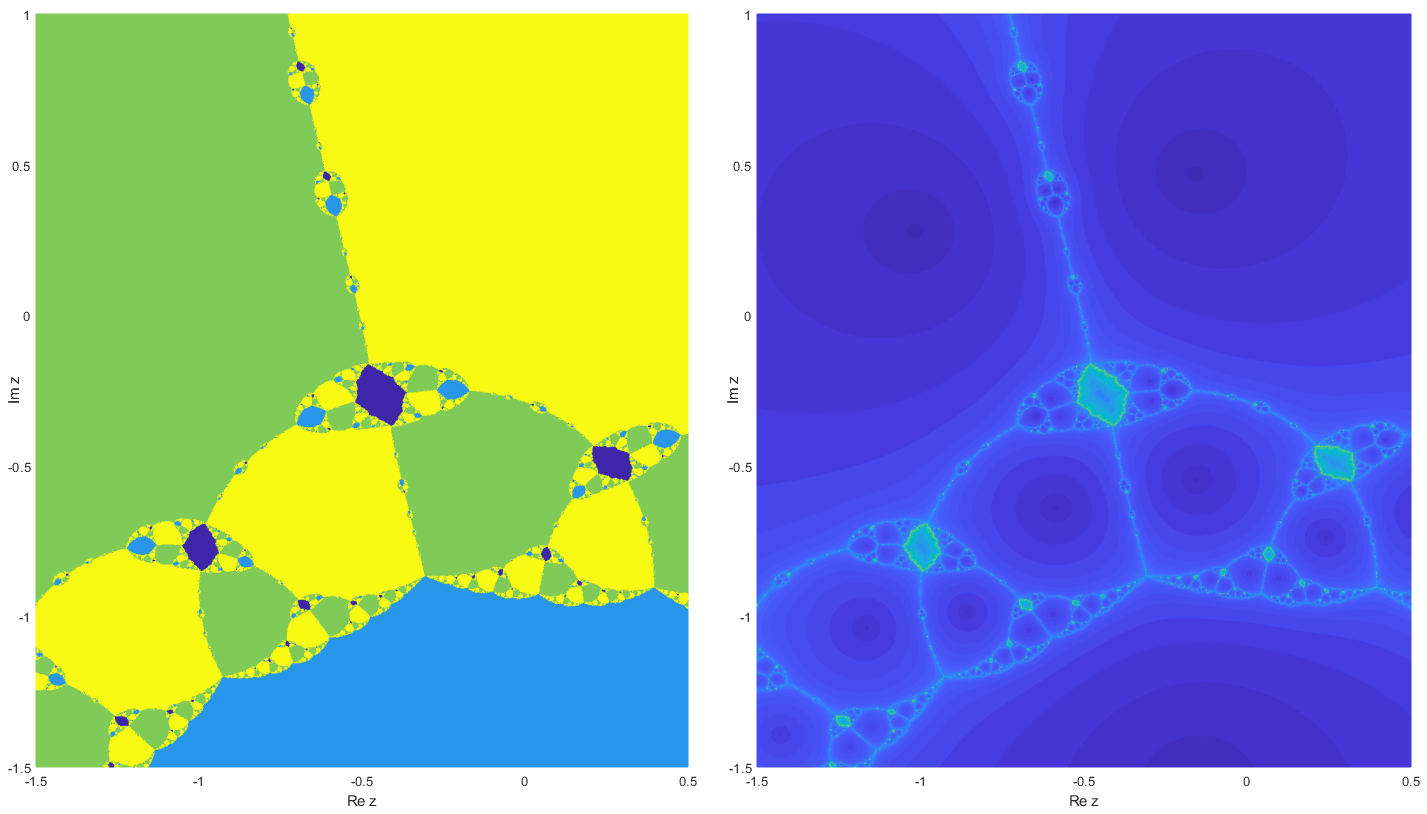}
\caption{Polynomial example. Left: attraction basins of the three minima (green, yellow, light-blue) and the periodic 2-cycle (blue). Right: number of iterations needed for $\varepsilon$-convergence to an attracting object.}
\label{fractal_cubic}
\end{figure}

\section{Application in telecommunications} \label{sec:applications}

In this section we describe the application of the MNM in parameter estimation as it is used in telecommunications.

\subsection{Linear system identification}

In this section, we consider a very simple practical example for linear system identification, that is widely used in telecommunications for channel estimation, linear equalization, prediction and also for adaptive noise cancellation.

The linear system can be described in matrix form as $y = Uz$, here $y\in \mathbb{C}^n$ is the output vector, $U\in \mathbb{C}^{n\times m}$ is a matrix which contains delayed input signal vectors, $z\in \mathbb{C}^m$ is the system parameter vector. The identification task is to estimate the coefficients $z$ by minimization of the 2-norm of the difference between the system output and a desired signal $d\in\mathbb{C}^n$:

\begin{equation*}
  	\min_{z \in D} f(z):= (d - Uz)^* (d-Uz) = \sum_{j = 1}^n \left|d_j - u_j^T z \right|^2,
\end{equation*}
where $u_j^T$ is the $j$-th row of the state matrix $U$.

By Lemma \ref{lem:complex_affine} the MNM solves this optimization problem in a single iteration, just as the ordinary Newton method with full step. However, the MNM step requires fewer operations since it solves a smaller linear system.

%
%
%
%

\subsection{MNM for two layer Hammerstein model estimation}

We now consider a more complicated non-linear problem. The two-layer non-linear Hammerstein model is widely used for power amplifier behaviour modelling in telecommunications systems. 

Graphically the Hammerstein model can be presented as a two layer cascade of a non-linear function and a linear convolution as shown on Fig.~\ref{hammerstin_model}.
\begin{figure}[h]
	\centering
	\includegraphics[]{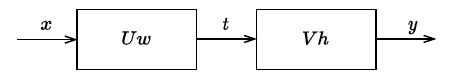}
	\caption{Two layer Hammerstein model structure}
	\label{hammerstin_model}
\end{figure}
The state matrix $U \in \mathbb{C}^{n \times (r+1)}$ contains non-linear functions $\varphi_p(x_{j})$ applied to the input signal $x$. The vector $t \in \mathbb{C}^n$ is the non-linear output of this first block. Delayed non-linear outputs $t_{j-q}$ generate the state matrix $V\in \mathbb{C}^{n \times (m+1)}$ in the second layer.

The $j$-th element $y_j$ of the model output can hence be described as
\begin{equation}
	y_j = \sum_{q = 0}^{m} h_q \underbrace{\sum_{p = 0}^{r} \varphi_p(x_{j-q}) w_p}_{\text{non-linear function}}  = \sum_{p = 0}^{r} w_p \underbrace{\sum_{q = 0}^{m} h_q \varphi_p(x_{j-q})}_{\text{linear convolution}},\quad j = 1\ldots n,
\end{equation}
where $x_{j}$ is the $j$-th element of the input signal $x$, $\varphi_p(\cdot)$ is the $p$-th non-linear basic function, for example a Chebyshev polynomial of order $p$, $w_p$ is the $p$-th basic function coefficient, $h_q$ is the $q$-th linear filter coefficient. 

The Hammerstein model estimation problem can be formulated as
\begin{equation}
 	\min_{w,h \in D} f(w,h):= (d - y)^* (d-y) = \sum_{j = 1}^n \left|d_j - \sum_{q = 0}^{m} h_q \sum_{p = 0}^{r} \varphi_p(x_{j-q}) w_p \right|^2.
 	\label{hamm_cost}
\end{equation}

In matrix form the Hammerstein model output can be written as
\begin{equation*}
	y = Vh = \tilde U w, \quad V_{j,q} =  \sum_{p = 0}^{r} \varphi_p(x_{j-q}) w_p, \quad \tilde{U}_{j,p} =  \sum_{q = 0}^{m} h_q \varphi_p(x_{j-q})
\end{equation*}
The bilinearity of the model in the parameters $h,w$ leads to non-convexity of the problem. If $w_{\text{opt}}$ and $h_{\text{opt}}$ are the optimal coefficients, then $w_{\text{opt}}/\alpha$ and $h_{\text{opt}}\alpha$ will provide the same cost function value for any non-zero complex constant $\alpha$.

We can concatenate the coefficients $w$ and $h$ in one common vector $z \in \mathbb{C}^{r+m+2}$ and calculate the first and mixed second derivatives appearing in the MNM. The first derivative is given by

\begin{equation*}
	\frac{\partial f}{\partial \bar z}  =
	\begin{pmatrix}
		\frac{\partial f}{\partial \bar w}\\
		 \frac{\partial f}{\partial \bar h}
	\end{pmatrix} =
	\begin{pmatrix}
		\frac{\partial f}{\partial \bar w} (d - \tilde U w)^*(d - \tilde U w)\\
		\frac{\partial f}{\partial \bar h} (d - V h)^*(d - V h)
	\end{pmatrix} =
	\begin{pmatrix}
		-\tilde U^*(d - \tilde U w)\\
		-V^*(d - V h)
	\end{pmatrix}.
\end{equation*}

The mixed derivative can be calculated if we take into account the property $Vh = \tilde U w$:
\begin{equation*}
	\frac{\partial f^2}{\partial \bar z \partial z}  =
	\begin{pmatrix}
		\frac{\partial f^2}{\partial \bar w \partial w} & \frac{\partial f^2}{\partial \bar w \partial h}\\
		\frac{\partial f^2}{\partial \bar h \partial w}  & \frac{\partial f^2}{\partial \bar h \partial h}
	\end{pmatrix} =
	\begin{pmatrix}
		\tilde U^* \tilde U & \tilde U^* V\\
		V^* \tilde U & V^* V
	\end{pmatrix}.
\end{equation*}
Finally the $k-$th iteration of the MNM estimation of the Hammerstein model is:

\begin{equation*}
z_{k+1} = z_k  -
\begin{pmatrix}
	\tilde U^* \tilde U & \tilde U^* V\\
	V^* \tilde U & V^* V
\end{pmatrix}^{-1}
\begin{pmatrix}
	-\tilde U^*(d - \tilde U w_k)\\
	-V^*(d - V h_k)
\end{pmatrix}.
\label{hamm_iterations}
\end{equation*}
We did not use the regularization approach introduced in Lemma \ref{lem:regulMNM}. It is observed that the degeneracy issue becomes not relevant when not too close to the optimum. Clearly the mixed derivative is always a non-negative definite matrix despite the non-convex nature of the model. We can see also that $z = 0$ is a stationary point of the model and hence cannot be used as an initial point.

Testing of the MNM algorithm was performed for mobile station SKY66299-11 power amplifier nonlinear distortion  identification. The 4G signal $x$ was fed to the input of the power amplifier at 1950 MHz and the output of the amplifier $d$ was captured in the feedback circuit. The amplifier  in non-linear mode added the distortion $\xi = d-x$ to the original signal, which could be observed on the output. 

The amplifier structure includes a transistor in nonlinear mode and output matching circuits that add inertia or memory to the nonlinear distortion at the output of the transistor. The following two-layer Hammerstein model was used to describe the physical processes of the nonlinear amplifier:

\begin{equation*}
y_j =  \sum_{q = 0}^{m} h_q \sum_{p = 0}^{r} w_p x_{j-q} |x_{j-q}|^p, \quad m = 11, \quad p = 8, \quad j = 0,\ldots, n-1. 
\end{equation*}

The model parameters can be calculated by solving problem \eqref{hamm_cost} by virtue of the MNM iterations \eqref{hamm_iterations}. Initial values of the coefficients $w$ and $h$ were generated in the neighbourhood  of the saddle point $(w = 0,h = 0)$ as random complex vectors with different standard deviations from $10^{-6}$ to 1. For each standard deviation value, 100 initial vectors were generated and the MNM algorithm was started for each initial vector. The MNM convergence statistics is shown in the following table:

\begin{table}[h!]
\centering
\begin{tabular}{|l|l|l|l|l|l|l|}
\hline
$STD[w, h]$ & $N_{min}$ & $N_{max}$ & $N_{mean}$ & $NMSE_{min}$, dB & $NMSE_{max}$, dB& $Pr$, \% \\ \hline
$10^{-6}$   &  16       &  481      &    48.6    &   -38.645    &   -38.110    &    100          \\ \hline
$10^{-4}$   &  12       &  296      &    43.2    &   -38.645    &   -38.645    &    100          \\ \hline
$10^{-2}$   &  18       &  354      &    56.3    &   -38.645    &   -38.645    &    100          \\ \hline
$10^{0}$    &  15       &  784      &    45.4    &   -38.645    &   -27.543    &    99           \\ \hline
\end{tabular}
\end{table}

Here $STD[w, h]$ is the standard deviation of the initial coefficients $w$ and $h$;   $N_{min}$, $N_{max}$, and $N_{mean}$ are the minimum, maximum, and mean of the number of MNM iterations over 100 trials for this initial deviation, correspondingly; $NMSE_{min}$ and $NMSE_{max}$ are the minimum and maximum normalized square error over 100 trials; $Pr$ is the probability of convergence to a global minimum. The NMSE was calculated according to the equation:

\begin{equation*}
NMSE = 10 \log_{10} \frac{e^*e}{x^*x}, \quad e_j = \xi_j - y_j, \quad j = 0\ldots n-1,
\end{equation*}
where $e$ is the error vector between the captured distortions $\xi$ and the Hammerstein model output. A smaller NMSE value correponds to a better minimum.

We can see that the MNM algorithm converges to the global optimum (NMSE $\leq -38$~dB) if we start  from the neighbourhood of the zero saddle point even for a standard deviation $STD = 10^{-6}$. The algorithm needs more iterations to escape from the saddle point for some starting points, but the optimum is still achieved. The algorithm may converge to some local minimum if the starting point is far from the global minimum.

\begin{figure}[h]
    \begin{minipage}{0.5\textwidth}
     \frame{\includegraphics[width=0.95\linewidth]{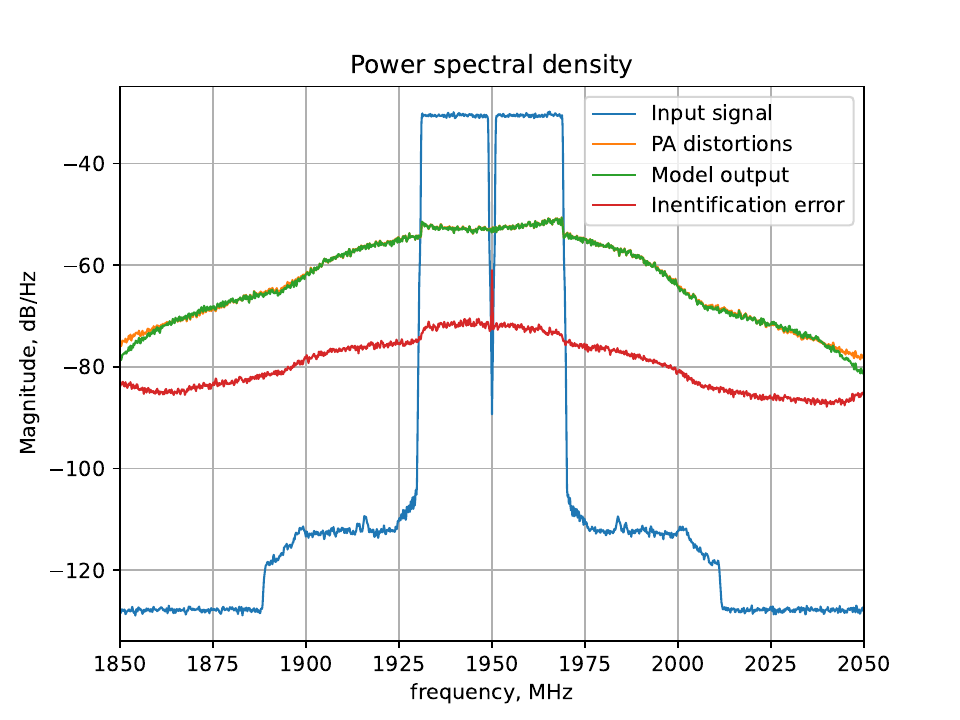}}
     \caption{Nonlinear power amplifier distortions identification}\label{hamm_psd}
   \end{minipage}
   \begin{minipage}{0.5\textwidth}
     \frame{\includegraphics[width=0.95\linewidth]{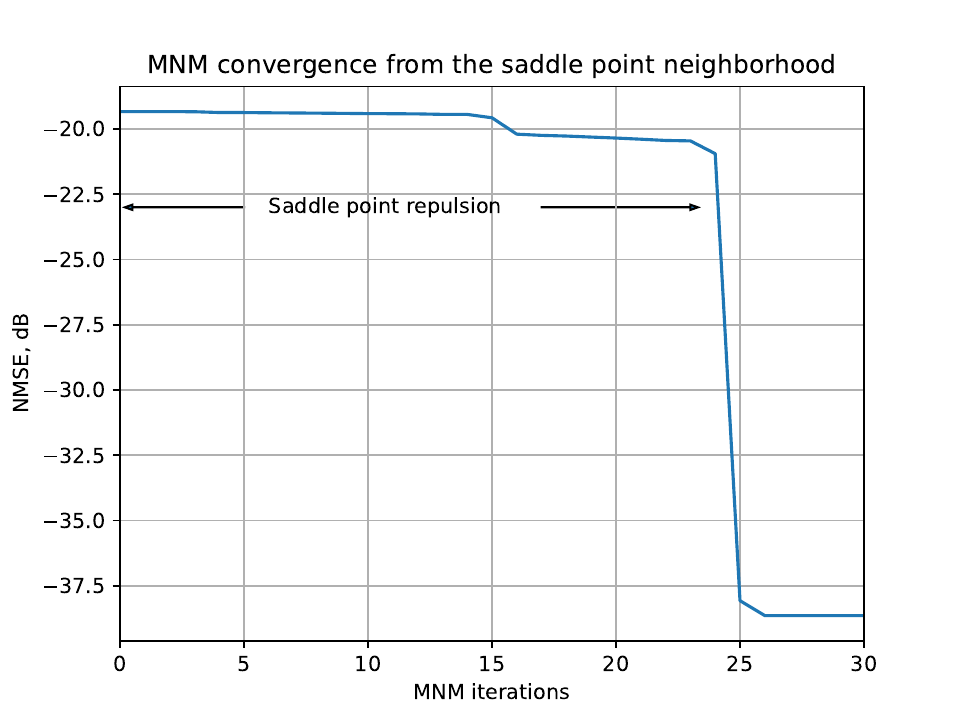}}
     \caption{MNM convergence from the saddle point neighbourhood}\label{hamm_conv}
   \end{minipage}
\end{figure}

The MNM convergence process is shown on Fig.~\ref{hamm_conv} for starting from the saddle point neighbourhood with $STD = 10^{-6}$. Observe that during the first 22 \ldots 24 iterations the MNM escapes from the saddle point and after this converges to the global minimum quadratically during 3 \ldots 5 steps.

\section{Conclusion} \label{sec:conclusion}

In this work, we propose and study an alternative to the Newton method for minimization of functions of complex variables $z$ of the form \eqref{f_definition}. Instead of using the full complex Hessian, only the mixed derivative with respect to $z,\bar z$ is computed at each iteration. The method is less demanding computationally and has superior global convergence properties, both theoretically and in practice. In particular, its iterates are attracted to minima, while they are repelled from saddle points. 

While the Newton method builds and minimizes a quadratic Taylor approximation of the objective, in the MNM builds the quadratic approximation by linearizing the holomorphic functions $g_j$ in \eqref{f_definition}. For the case of one scalar holomorphic function, the method reduces to the well-known complex Newton method for the search of zeros and exhibits a fractal structure of the attraction basins.

We illustrate the properties of the method with a number of examples. For the use in applications in wireless communications, a degeneracy problem arises which, as we show, can be surmounted by a regularization procedure.

\appendix

\section{Wirtinger derivatives} \label{sec:Wirtinger}

In this section, we provide basic notions on Wirtinger derivatives and gradients and Hessians of functions in the complex variables $z,\bar z$. For further information we refer to \cite[Section 1.4]{Remmert}. A more detailed review of the application to optimization can be found in \cite[Section 2]{SorberBarelLathauwer12}.

Consider a real-valued function $f(z)$ of a complex vector-valued variable $z \in D \subset \mathbb C^n$. The canonical way to compute the derivatives of $f$ is to consider it as a function of $2n$ real-valued variables, organized into two vectors $Re\,z$, $Im\,z$ of dimension $n$ each. This gives rise to the gradient and Hessian
\begin{equation} \label{realGradHess}
\nabla_{\mathbb R}f = \begin{pmatrix} \frac{\partial f}{\partial Re\,z} \\ \frac{\partial f}{\partial Im\,z} \end{pmatrix}, \qquad H_{\mathbb R}f = \begin{pmatrix} \frac{\partial^2 f}{\partial (Re\,z)^2} & \frac{\partial^2 f}{\partial Re\,z \partial Im\,z} \\ \frac{\partial^2 f}{\partial Im\,z \partial Re\,z} & \frac{\partial f}{\partial (Im\,z)^2} \end{pmatrix}.
\end{equation}

The Wirtinger derivatives are obtained by a formal complex-linear change of coordinates. Consider $z = Re\,z + iIm\,z$, $\bar z = Re\,z - iIm\,z$ as new independent coordinates, then we obtain
\begin{equation} \label{cr}
\nabla_{\mathbb C}f = \begin{pmatrix} \frac{\partial f}{\partial z} \\ \frac{\partial f}{\partial \bar z} \end{pmatrix} = \begin{pmatrix} I & iI \\ I & -iI \end{pmatrix}^{-T}(\nabla_{\mathbb R}f), \qquad H_{\mathbb C}f = \begin{pmatrix} \frac{\partial^2f}{\partial z^2} & \frac{\partial^2f}{\partial z\partial\bar z} \\ \frac{\partial^2f}{\partial\bar z\partial z} & \frac{\partial^2f}{\partial\bar z^2} \end{pmatrix} = \begin{pmatrix} I & iI \\ I & -iI \end{pmatrix}^{-T}(H_{\mathbb R}f)\begin{pmatrix} I & iI \\ I & -iI \end{pmatrix}^{-1}.
\end{equation}
Here the coefficient matrix whose inverse pre- and post-multiplies the real derivatives is defined by the Jacobian $\frac{\partial(z,\bar z)}{\partial(Re\,z,Im\,z)}$.

More generally, the Wirtinger derivatives are explicitly defined by the action of the operators
\begin{equation} \label{complex_diff_operators}
\frac{\partial}{\partial z} = \frac12\left( \frac{\partial}{\partial Re\,z} - i\frac{\partial}{\partial Im\,z} \right), \qquad \frac{\partial}{\partial \bar z} = \frac12\left( \frac{\partial}{\partial Re\,z} + i\frac{\partial}{\partial Im\,z} \right),
\end{equation}
where $\frac{\partial}{\partial Re\,z}$, $\frac{\partial}{\partial Im\,z}$ are vectors of $n$ partial derivatives defined with respect to the corresponding coordinates in real space $\mathbb R^{2n} \sim \mathbb C^n$.

In particular, for a real-valued function $f$ it follows that
\begin{equation} \label{Wirtinger_conjugate}
\frac{\partial f}{\partial z} = \overline{\frac{\partial f}{\partial\bar z}},
\end{equation}
i.e., the derivatives $\frac{\partial f}{\partial z},\overline{\frac{\partial f}{\partial\bar z}}$ are complex conjugates of each other. We also have that if $g$ is holomorphic, then
\begin{equation} \label{Wirtinger_holomorphic}
\frac{\partial g}{\partial z} = g', \qquad \frac{\partial g}{\partial\bar z} = 0,
\end{equation}
where $g' = \frac{dg}{dz}$ is the vector of ordinary holomorphic derivatives of $g$. This is a well-known consequence of the Cauchy-Riemann equations linking the partial derivatives of the real and imaginary parts of a holomorphic function. Likewise, for an anti-holomorphic function (i.e., whose complex conjugate is holomorphic) $u$ we have
\begin{equation} \label{Wirtinger_antiholomorphic}
\frac{\partial u}{\partial z} = 0, \qquad \frac{\partial u}{\partial\bar z} = \overline{(\bar u)'}.
\end{equation}

Together with the Leibniz product rule we obtain
\begin{equation} \label{modulus_derivatives}
\frac{\partial|g|^2}{\partial z} = g'\cdot \bar g, \quad \frac{\partial|g|^2}{\partial \bar z} = g \cdot \overline{g'}, \quad \frac{\partial^2|g|^2}{\partial\bar z\partial z} = \overline{g'} \cdot (g')^T
\end{equation}
for the Wirtinger derivatives of the squared modulus of a holomorphic function $g$.

\section{Funding}

The research leading to these results received funding from Huawei Technologies Co. Ltd. and the
Ministry of Science and Higher Education of the Russian Federation (Goszadaniye) project No. FSMG-2024-0011.

\bibliographystyle{plain}
\bibliography{interior_point,fractals,complex}

\begin{thebibliography}{10}

\bibitem{AdaliSchreier14}
T.~Adali and P.~J. Schreier.
\newblock Optimization and estimation of complex-valued signals: Theory and
  applications in filtering and blind source separation.
\newblock {\em IEEE Signal Process. Mag.}, 31(5):112--128, 2014.

\bibitem{Beardon}
Alan~F. Beardon.
\newblock {\em Iteration of Rational Functions. Complex Analytic Dynamical
  Systems}, volume 132 of {\em Graduate Texts in Mathematics}.
\newblock Springer, 1991.

\bibitem{Candan19}
Catagay Candan.
\newblock Properly handling complex differentiation in optimization and
  approximation problems.
\newblock {\em IEEE Signal Processing Magazine}, 36(2):117--124, 2019.

\bibitem{Cayley1879}
M.A. Cayley.
\newblock Application of the {N}ewton-{F}ourier method to the imaginary root of
  an equation.
\newblock {\em Quart. J. Math.}, 16:179--185, 1879.

\bibitem{CurryGarnettSullivan83}
James~H. Curry, Lucy Garnett, and Dennis Sullivan.
\newblock On the iteration of a rational function: Computer experiments with
  {N}ewton's method.
\newblock {\em Commun. Math Phys.}, 91:267--277, 1983.

\bibitem{HannaSimaan85}
M.T. Hanna and M.~Simaan.
\newblock A closed-form solution to a quadratic optimization problem in complex
  variables.
\newblock {\em J. Optim. Theory Appl.}, 47(4):437--450, 1985.

\bibitem{NesNem94}
Yurii Nesterov and Arkadii Nemirovskii.
\newblock {\em Interior-point Polynomial Algorithms in Convex Programming},
  volume~13 of {\em SIAM Stud. Appl. Math.}
\newblock SIAM, Philadelphia, 1994.

\bibitem{PeitgenSaupeHaeseler84}
Heinz-Otto Peitgen, Dietmar Saupe, and Fritz von Haeseler.
\newblock Cayley's problem and {J}ulia sets.
\newblock {\em The Mathematical Intelligencer}, 6(2):11--20, 1984.

\bibitem{Remmert}
Reinhold Remmert.
\newblock {\em Theory of Complex Functions}, volume 122 of {\em Graduate Texts
  in Mathematics}.
\newblock Springer, New York, 1991.

\bibitem{SchreierScharf}
P.~J. Schreier and L.~L. Scharf.
\newblock {\em Statistical Signal Processing of Complex-Valued Data}.
\newblock Cambridge Univ. Press, Cambridge, 2010.

\bibitem{SorberBarelLathauwer12}
Laurent Sorber, Marc van Barel, and Lieven de~Lathauwer.
\newblock Unconstrained optimization of real functions in complex variables.
\newblock {\em SIAM J. Optim.}, 22(3):879--898, 2012.

\bibitem{Bos94}
A.~van~den Bos.
\newblock Complex gradient and {H}essian.
\newblock {\em Proc. Inst. Elect. Eng.}, 141(6):380--382, 1994.

\bibitem{Wiersma16}
Alida~H. Wiersma.
\newblock The complex dynamics of {N}ewton's method.
\newblock Master's thesis, Groningen University, 2016.

\bibitem{YanFan00}
G.~Yan and H.~Fan.
\newblock A {N}ewton-like algorithm for complex variables with applications in
  blind equalization.
\newblock {\em IEEE Trans. Signal Process.}, 48(2):553--556, 2000.

\end{thebibliography}

\end{document}